\documentclass[11pt]{amsart}
\usepackage[utf8]{inputenc}
\usepackage[english]{babel}
\usepackage[T1]{fontenc}
\usepackage{amsthm}
\usepackage{amssymb}
\usepackage{amsmath}
\usepackage{mathtools}
\usepackage{pdfpages}
\usepackage{changes}
\usepackage{tikz-cd}
\usepackage{soul}
\usepackage{float}
\usepackage{multicol}

\newcommand{\N}{\mathbb{N}}
\newcommand{\Z}{\mathbb{Z}}

\newcommand{\Q}{\mathbb{Q}}

\newcommand{\D}{\mathbb{D}}
\newcommand{\Sym}{\mathbb{S}}
\newcommand{\Alt}{\mathbb{A}}
\newcommand{\Aut}{\operatorname{Aut}}

\newcommand{\GL}{\mathbf{GL}}
\newcommand{\id}{\mathrm{id}}
\newcommand{\Soc}{\mathrm{Soc}}
\newcommand{\Fix}{\mathrm{Fix}}

\newcommand{\Ret}{\operatorname{Ret}}

\makeatletter
\numberwithin{equation}{section}
\numberwithin{figure}{section}
\numberwithin{table}{section}
\newtheorem{thm}{Theorem}[section]
\newtheorem{lem}[thm]{Lemma}
\newtheorem{cor}[thm]{Corollary}
\newtheorem{pro}[thm]{Proposition}
\newtheorem{defn}[thm]{Definition}

\newtheorem{notation}[thm]{Notation}

\newtheorem{convention}[thm]{Convention}
\newtheorem{algorithm}[thm]{Algorithm}
\newtheorem{rem}[thm]{Remark}
\newtheorem{exa}[thm]{Example}

\makeatother

\def\Z{{\mathbb{Z}}}

\keywords{Bieberbach group, Yang-Baxter equation, set-theoretic solution, multipermutation solution, unique product property, skew brace}

\title[Retractability of solutions]{Retractability of solutions to the Yang--Baxter equation and $p$-nilpotency of skew braces}

\begin{document}

\begin{abstract}
	Using Bieberbach groups we study multipermutation involutive solutions to
	the Yang--Baxter equation. We use a linear representation of the structure
	group of an involutive solution to study the unique product property in
	such groups. An algorithm to find subgroups of a Bieberbach group
	isomorphic to the Promislow subgroup is introduced and then used in the
	case of structure group of involutive solutions. To extend the results
	related to retractability to non-involutive solutions, following the ideas
	of Meng, Ballester-Bolinches and Romero, we develop the theory of right
	$p$-nilpotent skew braces. The theory of left $p$-nilpotent skew braces is
	also developed and used to give a short proof of a theorem of Smoktunowicz
	in the context of skew braces. 
\end{abstract}

\author{E. Acri}
\author{R. Lutowski}
\author{L. Vendramin}

\address[E. Acri, L. Vendramin]{IMAS--CONICET and Universidad de Buenos Aires, 
Pabell\'on~1, Ciudad Universitaria, 1428, Buenos Aires, Argentina}
\email{eacri@dm.uba.ar}

\address[R. Lutowski]{Institute of Mathematics, University of Gdańsk, ul. Wita Stwosza 57, 80-952 Gdańsk, Poland}
\email{rafal.lutowski@mat.ug.edu.pl}

\address[L. Vendramin]{NYU-ECNU Institute of Mathematical Sciences at NYU Shanghai, 3663 Zhongshan Road North, Shanghai, 200062, China}
\email{lvendramin@dm.uba.ar}

\maketitle

%\noindent{\em 2010 MSC: 16T25,81R50.}

\section*{Introduction}

In order to construct solutions to the celebrated Yang--Baxter equation, Drinfeld introduced in~\cite{MR1183474} set-theoretic solutions, i.e. pairs $(X,r)$ where $X$ is a set and $r\colon X\times X\to X\times X$ is a bijective map such that
\[
(r\times\id)(\id\times r)(r\times\id)=
(\id\times r)(r\times\id)(\id\times r).
\]

To address this problem by means of combinatorial methods, one considers non-degenerate
solutions, i.e. solutions $(X,r)$ where the bijective map $r$ can be written as 
\[
r(x,y)=(\sigma_x(y),\tau_y(x)),\quad x,y\in X, 
\]
for permutations $\sigma_x\colon X\to X$ and $\tau_x\colon X\to X$. 
An example of a non-degenerate solution is that of 
Lyubashenko, where the map $r$ is given by 
$r(x,y)=(\sigma(y),\tau(x))$ for $\sigma$ and $\tau$ commuting permutations of $X$. 

The first papers on set-theoretic solutions are those of Etingof, Schedler and Soloviev~\cite{MR1722951} and Gateva--Ivanova and Van den Bergh~\cite{MR1637256}. 
Both papers considered involutive solutions, i.e. solutions $(X,r)$ where $r^2=\id$. 

In~\cite{MR2278047}, 
Rump observed that each radical ring $R$ produces an involutive solution. 
(A \emph{radical ring} $R$ is a ring such that the Jacobson circle operation $x\circ y=x+y+xy$ turns $R$ into a group.) Then he introduced a new algebraic structure 
that generalizes radical rings and provides an algebraic framework to study 
involutive solutions. This new structure showed connections between the Yang--Baxter equation
and ring theory, flat manifolds, orderability of groups, Garside theory, see for example~\cite{MR3815290,MR3572046,MR3190423,MR3374524,MR2927367,MR2776789,MR3291816}.

As a tool to construct involutive solutions, Etingof, Schedler and Soloviev introduced  
retractable solutions~\cite{MR1722951}. Such solutions are those that induce 
a smaller solution after identifying certain elements of the underlying set. 
Multipermutation solutions are then those solutions that can be retracted to the trivial solution 
over the set with only one element after a finite number of steps. This means 
that multipermutation solutions generalize those solutions of Lyubashensko.
Several papers study
multipermutation involutive solutions, see for example~\cite{MR3719300,MR3780533,MR2652212,MR3861714,MR2885602,MR3439888,MR3765444,MR3771874}. 

Almost all of the ideas used in the theory of involutive solutions can be 
transported to non-involutive solutions. The algebraic framework now is provided
by skew braces~\cite{MR3647970}. This rich structure shows that the Yang--Baxter equation
is related to different topics such as Hopf--Galois extensions, regular subgroups and nil-rings. For that reason, the theory of (skew) braces is intensively studied, see for example~\cite{MR3835326,trusses,MR3478858,MR3649817,MR3574204,cjo,MR3177933,MR3834774,reflection}.

This paper explores the retractability problem for solutions
and their applications to the theory of (skew) braces. In the case of finite involutive solutions, this is done
by using in different ways the fact that the structure group of the solution is a Bierberbach group. Since 
different methods are needed 
to obtain similar results for non-involutive solutions, 
we follow the ideas of Meng, Ballester--Bolinches and Romero~\cite{mbbr} and develop   
the theory of right $p$-nilpotent skew left braces. We also 
study left $p$-nilpotent skew left braces and 
(again following Meng, Ballester--Bolinches and Romero)  
we give a short proof of a theorem of Smoktunowicz~\cite[Theorem 1.1]{MR3814340} related to left nilpotency in the context of skew left braces, 
see~\cite[Theorem 4.8]{csv}. 

The paper is organized as follows.  Section~\ref{preliminaries} is devoted to
preliminaries on set-theoretic solutions to the Yang--Baxter equation and the
theory of skew braces. In Section~\ref{Bieberbach} we recall a faithful
linear representation of the structure group of a finite involutive solution
constructed by Etingof, Schedler and Soloviev.  Structure groups of finite
involutive solutions are Bieberbach groups, and with
the faithful linear representation constructed we compute explicitly the
holonomy group of this Bieberbach group. These results are then applied to the
retractability problem of involutive solutions. In Section~\ref{UPP} we study
the unique product property for structure groups of involutive solutions. We
prove that all structure groups of involutive solutions of size $\leq7$ that
are not multipermutation solutions do not have the unique product property. In
Section~\ref{Promislow} we present an algorithm that detects subgroups of an
arbitrary Bieberbach group that are isomorphic to the Promislow group; this
algorithm is then used in Theorem~\ref{thm:P} to prove that all but eight
structure groups of involutive solutions of size $\leq8$ that are not
multipermutation solutions do not have the unique product property. To extend
some of our results to non-involutive solutions different methods are needed.
Following the ideas of Meng, Ballester--Bolinches and Romero~\cite{mbbr}, we
introduce right $p$-nilpotency of skew left braces of nilpotent type and use
this concept to explore retractable non-involutive solutions in
Section~\ref{right}. Finally, in Section~\ref{left}, again
following~\cite{mbbr}, we study left $p$-nilpotent skew left braces. 

\section{Preliminaries}
\label{preliminaries}

\subsection{Set-theoretic solutions to the Yang--Baxter equation}
\label{YB}

A \emph{set-theoretic solution} to the Yang--Baxter equation (YBE) is a pair $(X,r)$, where $X$ is a set and $r\colon X\times X\to X\times X$ is a bijective map that satisfies 
\[
(r\times\id)(\id\times r)(r\times\id)=
(\id\times r)(r\times\id)(\id\times r).
\]
The solution $(X,r)$ is said 
to be \emph{finite} if $X$ is a finite set. By convention, we write
$r(x,y)=(\sigma_x(y),\tau_y(x))$. 
We say that $(X,r)$ is non-degenerate if the maps $\sigma_x$ and $\tau_x$ are 
permutations of $X$. 

\begin{convention}
By a \emph{solution} we will mean a non-degenerate solution to the YBE.
\end{convention}

The solution $(X,r)$ is said to be \emph{involutive} if $r^2=\id$. 
The \emph{structure group} $G(X,r)$ of $(X,r)$ is defined in~\cite{MR1722951,MR1769723,MR1809284}
as the group 
with generators $x\in X$ and relations
\[
x\circ y=u\circ v\quad\text{whenever $r(x,y)=(u,v)$}.
\]

If $(X,r)$ is finite involutive, the group $G(X,r)$ is torsion-free~\cite{MR1637256}. Moreover, $G(X,r)$ 
is a Garside group~\cite{MR2764830}; see~\cite{MR3374524} or~\cite{MR3824447} for other proofs. 

If $(X,r)$ is involutive, its \emph{permutation group} is the group
$\mathcal{G}(X,r)$ generated by the permutations $\sigma_x$ for $x\in X$. 
Clearly, $\mathcal{G}(X,r)$ acts on $X$ and 
$\mathcal{G}(X,r)$ is finite if $X$ is finite. 
The permutation group of a non-involutive solution was defined by Soloviev in~\cite{MR1809284}.

If $(X,r)$ is an involutive solution and $x,y\in X$,
following~\cite{MR1722951}, we say that $x\sim y$ if and only if
$\sigma_x=\sigma_y$. Then $\sim$ is an equivalence relation over $X$ that induces 
a solution $\Ret(X,r)$ over the set $X/{\sim}$. We define inductively
$\Ret^1(X,r)=\Ret(X,r)$ and 
$\Ret^{n+1}(X,r)=\Ret(\Ret^n(X,r))$ for $n\geq1$. An involutive
solution $(X,r)$ is said to be irretractable if $\Ret(X,r)=(X,r)$ and it is 
a \emph{multipermutation solution} if there exists $n$ such that 
$\Ret^n(X,r)$ has only one element. 
We refer to~\cite{JEDLICKA20193594,LV3,MR3763907} for some results related to the retractability of non-involutive solutions. 

\subsection{Skew left braces}
\label{skew}

We refer to~\cite{MR3647970} for the theory of skew left braces. 
A \emph{skew left brace} is a triple $(A, +, \circ)$, where $(A, +)$ and $(A, \circ)$ are 
groups such that
$a\circ (b + c) = a\circ b - a + a\circ c$ holds for all $a, b, c\in A$. 
We write $a'$ to denote the inverse
of the element $a\in A$ with respect to the circle operation. 
A skew left brace $A$ such that $a\circ b = a + b$ for all $a,b\in A$ is said to be \emph{trivial}. 
If $\mathcal{X}$ is a property of groups, a skew left brace is said to be of $\mathcal{X}$-type
if its additive group belongs to the class $\mathcal{X}$. Skew left braces of abelian type
are those braces introduced by Rump in~\cite{MR2278047} to study involutive solutions. 

\begin{convention}
Skew left braces of abelian type will be called \emph{left braces}.
\end{convention}

If $A$ is a skew left brace, 
the map $\lambda\colon (A,\circ)\to\Aut(A,+)$, $a\mapsto\lambda_a$, where
$\lambda_a(b)=-a+a\circ b$, is a group homomorphism. 
By definition, 
\begin{align*}
    &a\circ b=a+\lambda_a(b),
    &&a+b=a\circ\lambda^{-1}_a(b),
    &&\lambda_a(a')=-a.
\end{align*}
Moreover:
\begin{align*}
    &a*(b + c) = a*b + b + a*c - b,\\
    &(a\circ b) * c = a * (b * c) + b * c + a * c,
\end{align*}
where 
$a*b=\lambda_a(b)-b$.

The connection between skew left braces and the YBE is the following: If $A$ is a skew left brace, then the map
\[
r_A\colon A\times A\to A\times A,
\quad
r_A(a,b)=(\lambda_a(b),\lambda_a(b)'\circ a\circ b)
\]
is a solution of the YBE.
Moreover, $r_A^2=\id$ if and
only if $A$ is of abelian type. 

If $(X,r)$ is a solution, then the group
$G(X,r)$ has a unique skew left brace structure such that
\[
r_{G(X,r)\times G(X,r)}(\iota\times\iota)=(\iota\times\iota)r,
\]
where $\iota\colon X\to G(X,r)$ is the canonical map (which in general is not injective). Moreover, the skew left brace $G(X,r)$ satisfies a universal property: if $A$ is a skew left brace
and $f\colon X\to A$ is a map such that $r_A(f\times f)=(f\times f)r$, then there exists a unique skew left brace homomorphism $\varphi\colon G(X,r)\to A$ such that $\varphi\iota=f$ and $r_A(\varphi\times\varphi)=(\varphi\times\varphi)r_{G(X,r)}$. Similar results appear in a differently language
in~\cite{MR1722951,MR1769723,MR1809284}. 

Note that the multiplicative group of the skew left brace $G(X,r)$ is the structure group of $(X,r)$ defined in Subsection~\ref{YB}.

If $(X,r)$ is an involutive solution, then the permutation group $\mathcal{G}(X,r)$ is a left brace 
with additive structure given by 
\[
\lambda_a+\lambda_b=\lambda_a\lambda_{\lambda^{-1}_a(b)}
\]
for $a,b\in A$, 
see for example~\cite{MR3527540}. 
An analog result for non-involutive solutions is proved in~\cite{MR3835326}.

A \emph{left ideal} of a skew left brace is a subgroup of the additive group that is stable under the action of $\lambda$. It follows that a left ideal of a skew left brace is a subgroup of the multiplicative group of the
skew left brace. An \emph{ideal} of a skew left brace is a left ideal that is normal as a subgroup of the additive group and normal as a subgroup of the multiplicative group. A non-zero skew left brace is \emph{simple} if it has only two ideals. The \emph{socle} of a skew left brace $A$ is the ideal $\Soc(A)=\ker\lambda\cap Z(A,+)$, where $Z(A,+)$ denotes the center of the additive group of $A$. 

\begin{notation}
For a finite set $X$, $\pi(X)$ is the set of prime divisors of $|X|$.
\end{notation}

For subsets $X$ and $Y$ of a skew left brace $A$, we write $X*Y$ to denote the subgroup of $(A,+)$ generated by elements of the form $x*y$, where $x\in X$ and $y\in Y$. 

\begin{lem}
    \label{lem:Sylow:left_ideal}
    Let $A$ be a finite skew left brace of nilpotent type and $p\in\pi(A)$. 
    Each Sylow $p$-subgroup of $(A,+)$ is a left ideal of $A$. 
\end{lem}

\begin{proof}
    See for example~\cite[Lemma 4.10]{csv}.
\end{proof}

Lemma~\ref{lem:Sylow:left_ideal} only works for skew left braces of nilpotent type:

\begin{exa}
Let $G=\{g_j: j\in \mathbb{Z}/{6\mathbb{Z}}\}$. The operations
\[
g_i+g_j=g_{i+(-1)^ij},\quad
g_i\circ g_j=g_{i+j}
\]
turns $G$ into a skew left brace with multiplicative group isomorphic to the cyclic group $C_6$ and non-nilpotent additive group isomorphic to $\Sym_3$. 
The Sylow $2$-subgroups of $(G,+)$ are not left ideals of $G$.
\end{exa}

Let $G$ be a group and $p\in\pi(G)$ be such that $|G|=p^km$, where $p$ does not divide $m$. A Hall $p'$-subgroup of $G$ is a subgroup of order $m$.

\begin{lem}
\label{lem:Hall}
    Let $A$ be a finite skew left brace of nilpotent type. For each $p\in\pi(A)$, 
    the Hall $p'$-subgroup of $(A,+)$ given by
    \[
        A_{p'}=\sum_{q\in\pi(A)\setminus\{p\}}A_q,
    \]
    where each $A_q$ is the $q$-Sylow subgroup of $(A,+)$,
    is a normal subgroup of $(A,+)$ and it is a left ideal of $A$.
\end{lem}

\begin{proof}
    Since $(A,+)$ is nilpotent, $A_{p'}$ is a normal subgroup of $(A,+)$. Moreover, $A_{p'}$ is a left ideal of $A$ by Lemma~\ref{lem:Sylow:left_ideal} and the fact that the sum of left ideals is a left ideal.
\end{proof}

A skew left brace $A$ is said to be \emph{right nilpotent} 
if $A^{(n)}=0$ for some $n\geq1$, where
$A^{(1)}=A$ and $A^{(n+1)} = A^{(n)}*A$ for $n\geq1$. 
Each $A^{(n)}$ is an ideal of $A$.
In~\cite{MR2278047}, 
Rump introduced 
the sequence
\begin{equation}
    \label{eq:soc_sequence}
    0=\Soc_0(A)\subseteq\Soc_1(A)\subseteq\cdots\subseteq\Soc_n(A)\subseteq\cdots,
\end{equation}
and used it to study right nilpotent braces of abelian type and retractability of involutive solutions. In the context of skew left braces,~\eqref{eq:soc_sequence} is defined recursively as follows: $\Soc_0(A)=0$ and for each $n\geq 1$, $\Soc_{n+1}(A)$ is the ideal of $A$ containing $\Soc_n(A)$ and such that $\pi(\Soc_{n+1}(A))=\Soc(\pi(A))$, where $\pi\colon A\to A/\Soc_n(A)$ is the canonical map. 

For a skew left brace $A$ and $x,y\in A$, 
we write $[x,y]_+=x+y-x-y$ to denote the additive commutator of $x$ and $y$. 

\begin{lem}
\label{lem:soc_n}
Let $A$ be skew left brace. Then
\[
\Soc_{n+1}(A)=\{x\in A:x*a\in\Soc_n(A)\text{ and }[x,a]_+\in\Soc_n(A)\text{ for all 
$a\in A$}\}   
\]
for all $n\in\N$.
\end{lem}

\begin{proof}
    It is straightforward.
\end{proof}

\begin{lem}\label{lem:soc_n_2}
    Let $A$ be a skew left brace of nilpotent type. Then $A$ is right nilpotent if and only if $A=\Soc_n(A)$ for some $n\in\N$.
\end{lem}

\begin{proof}
    It follows from~\cite[Lemmas 2.15 and 2.16]{csv}.
\end{proof}

A skew left brace is said to be \emph{left nilpotent} if $A^n = 0$ for some $n$, where $A^1=A$ and $A^{n+1} = A * A^n$ for $n\geq1$. Each $A^n$ is a left ideal of $A$. We refer to~\cite{MR3574204,csv,mbbr,MR2278047,MR3765444,MR3814340, MR3763907} 
for results on left nilpotent skew left braces.

\section{Bieberbach groups}
\label{Bieberbach}

We refer to~\cite{MR2978307} for the theory of Bieberbach groups. 
A group $G$ is said to be an $n$-dimensional \emph{Bieberbach group} if it is
torsion free and contains an abelian normal subgroup $A\simeq\Z^n$ of
finite index such that $C_G(A)=A$, where
\[
C_G(A)=\{g\in G:ga=ag\text{ for all $a\in A$}\}.
\]
Thus $G$ fits into the exact sequence
\[
	0\to A\to G\xrightarrow{p}{} P\to 1,
\]
where $P=G/A$ is a finite group. The condition $C_G(A)=A$ is equivalent to the faithfulness of the action
$h\colon P\to \Aut(A)$, $h(y)(a)=xax^{-1}$, where $a\in A$ and
$x\in G$ is such that $p(x)=y$ induced by the conjugation action of $G$ over $A$. 
In the theory of
Bieberbach groups, $P$ is known as the \emph{holonomy group} of $G$, the map $h$ as the \emph{holonomy representation} of $G$ and $A$ as
the \emph{traslation subgroup} of $G$. 

%and let
%\[
%\Delta(G)=\{g\in G:\text{$g$ has only finitely many conjugates}\}
%\]
%be the finite conjugate subgroup of $G$. 
%The following lemma is classical, we provide a proof for completeness.
%goes back to Farkas, see
%page 533 of~\cite{farkas1981}.

%\begin{lem}
%\label{lem:Farkas}
%Let $A$ be a of finite-index, normal and torsion-free abelian subgroup of a group $G$. Then %$A=\Delta(G)$ if and only if $A=C_G(A)$. 
%\end{lem}

%\begin{proof}
%    Since $A$ is abelian by assumption, $A\subseteq C_G(A)$. 
%    Let $x\in C_G(A)$.
%    Then $x\in\Delta(G)=A$ since $A\subseteq C_G(x)$ and $A$ has finite index in $G$.

%    Now let $x\in\Delta(G)$.
%    Since $A$ and $C_G(x)$ have both finite index in $G$, the subgroup $A\cap C_G(x)$ has finite %index in $G$. Let $a\in A$ and $m\in\N$ be such that $a^m\in A\cap C_G(x)$. Since $x$ 
%    acts trivially on $A\cap C_G(A)$, $(xax^{-1})^m=a^m$. Since $A$ has no torsion, this implies %that $xax^{-1}=a$. 
%\end{proof}

%\begin{lem}
%	\label{lem:maximal_abelian}
%	Let $H$ be a subgroup of $G$. Then $H=C_G(H)$ if and only if $H$ is a
%	maximal abelian subgroup of $G$.
%\end{lem}
%
%\begin{proof}
%	Since $H$ is abelian, $H\subseteq C_G(H)$. If $H\subsetneq C_G(H)$, let
%	$x\in C_G(H)\setminus H$. Then $\langle H,x\rangle$ is an abelian subgroup
%	of $G$ containing $H$, a contradiction. The converse is trivial.
%\end{proof}

The group $G$ can be seen as a discrete subgroup of the isometries of a finite-dimensional euclidean space, that is $G\subseteq \mathcal{O}_n(\mathbb{R})\ltimes \mathbb{R}^n$ for some $n$. In this case, the \emph{translation subgroup} $A$ can be seen as $G\cap \mathbb{R}^n$, see \cite[page 533]{farkas1981}. 
%This observation will be useful in what follows.

In~\cite[Theorem 1.6]{MR1637256}, Gateva--Ivanova and Van den Bergh proved that if $(X,r)$ is a finite involutive solution, then the structure group of $G(X,r)$ is a Bieberbach group of dimension $|X|$. 
The holonomy group of $G(X,r)$ will be computed in Theorem \ref{thm:holonomy}.
First we need a faithful representation of $G(X,r)$ that allows us to deal with these groups as subgroups of the isometries of an euclidean space. 
The following result goes back to Etingof, Schedler and Soloviev,
see~\cite{MR1722951}. 

\begin{thm}
    \label{thm:ESS}
	Let $(X,r)$ be a finite involutive solution of size $n$.  Then there exists
	an injective group homomorphism $G(X,r)\to \mathcal{O}_n(\mathbb{R})\ltimes
	\mathbb{R}^n$.  In particular, $G(X,r)$ is isomorphic to a subgroup of
	$\GL(n+1,\mathbb{Z})$.
\end{thm}

\begin{proof}
    Let $\Sym_X$ denote the group of permutations of $X$ and let $\Z^X$ be
    the free abelian group spanned by $\{t_x:x\in X\}$. 
    Let $M_X=\Sym_X\ltimes\Z^X$ be the semidirect product associated with
    the action of $\Sym_X$ on $\Z^X$. By Propositions 2.3 and 2.4 of~\cite{MR1722951}, the map $X\to M_X$, $x\mapsto (\sigma_x,t_x)$, extends to an injective group
    homomorphism $G(X,r)\to M_X$. Using permutation matrices we see $\Sym_X$ as a subgroup of $\mathcal{O}_n(\mathbb{Z})\subseteq \mathcal{O}_n(\mathbb{R})$. Then, since $\Z^X\simeq \Z^n\subseteq\mathbb{R}^n$, it follows that $M_X$ is isomorphic to a subgroup of the semidirect product $\mathcal{O}_n(\mathbb{R})\ltimes \mathbb{R}^n$. Since the multiplication
    of $\mathcal{O}_n(\mathbb{R})\ltimes \mathbb{R}^n$ is given by
    \[
    (A,a)(B,b)=(AB,a+Ab),
    \]
    after identifying each $(A,a)\in M_X$ with the matrix 
    $\left(\begin{smallmatrix}
    A & a\\
    0 & 1
    \end{smallmatrix}\right)\in\GL(n+1,\Z)$, the claim follows.
\end{proof}

Notice that under this identification, we can see at the socle of $G(X,r)$ as the translation subgroup, i.e. it is the set of elements of $\mathcal{O}_n(\mathbb{R})\ltimes \mathbb{R}^n$ of the form $(I,a)\in M_X \subset \GL(n+1,\Z)$. Furthermore, 
    \[
    C_{G(X,r)}(\Soc(G(X,r)))=\Soc(G(X,r)).
    \]
    Since $\Soc(G(X,r))$ is abelian, 
    \[
    \Soc(G(X,r))\subseteq C_{G(X,r)}(\Soc(G(X,r))).
    \]
    Now for every $(I,x)\in\Soc(G(X,r))$ and $(A,a)\in C_{G(X,r)}(\Soc(G(X,r)))$ we have
    %\marginpar{RL: It is not enough to take one element $(I,x)$ - you must take those $x$'s which generate %$\mathbb{R}^n$}
    \[
    (A,a)(I,x)(A^{-1},-A^{-1}a)=(I,Ax) = (I,x),
    \]
    so $Ax=x$ holds for all elements of the set $\{x\in \mathbb{R}^n : (I,x) \in \Soc(G(X,r))\}$. But by the first Bieberbach theorem (see \cite[Theorem 2.1]{MR2978307}) this set spans $\mathbb{R}^n$, hence $(A,a)\in C_{G(X,r)}(\Soc(G(X,r)))$ if and only if $A=I$. Thus the only elements of the group that centralizes the socle are exactly the elements of the socle.
    
    We know from \cite[Theorem 1.6]{MR1637256} that the structure group of a solution is Bieberbach. As a direct consequence of the first Bieberbach Theorem, the socle is the subgroup of pure translations and it is torsion-free and maximal normal abelian subgroup of finite index. So, the holonomy group is exactly the permutation group of the solution. The holonomy representation $h$ is the action by conjugation of $G(X,r)$ over the socle that descends to a faithful representation. We summarize this result in the following theorem for the sequel.
    
\begin{thm}
	\label{thm:holonomy}
	Let $(X,r)$ be a finite involutive solution. Then $G(X,r)$ is a Bieberbach group with holonomy group isomorphic to $\mathcal{G}(X,r)$.
\end{thm}

% \begin{proof}
%     \deleted{We only need to prove that the holonomy group of the multiplicative
%     group $G$ of the left brace $G(X,r)$ is isomorphic to
% 	$\mathcal{G}(X,r)$.  
%     The group $G$ is torsion-free by~\cite[Theorem 1.6]{MR1637256}. Let $A$ be the multiplicative
%     group of the socle of $G(X,r)$. Then $A$ is a torsion-free abelian normal subgroup of finite index in $G$.
%     We claim that $C_G(A)=A$. Since $A$ is abelian, $A\subseteq C_G(A)$. Since $G$ acts by conjugation in $A$, the quotient $G/A\simeq\mathcal{G}(X,r)$
%     acts faithfully on $A$.
%     Assume that $A\ne C_G(A)$ and let $g\in
% 	C_{G}(A)\setminus A$. Then $p(g)\ne 1$. Since 
% 	$h(p(g))(a)=gag^{-1}=a$ for all $a\in A$, we conclude that $h(p(g))=\id$, a contradiction to the
% 	faithfulness of $h$. Therefore $C_G(A)=A$.}
% \end{proof}

\subsection{Applications to the YBE}

Multipermutation solutions are related to orderability of groups.  Jespers and
Okni\'nski proved in~\cite[Proposition 4.2]{MR2189580} that the structure group
of a finite involutive multipermutation solution is poly-$\Z$ and hence left orderable.
Independently in~\cite[Theorem 2]{MR3572046} Chouraqui, interested in 
studying left orderability of structure groups of involutive solutions, proved the same result. It was
proved later in~\cite[Theorem 2.1]{MR3815290} that a finite involutive solution is multipermutation if and only if its structure group is left orderable. 
A group $G$ is said to be \emph{diffuse} if for each finite non-empty subset $A$ of $G$ there exists
an element $a\in A$ such that for all $g\in G$, $g\ne1$, either $ga\not\in A$ or $g^{-1}a\not\in A$. 
In~\cite[Theorem 7.12]{LV3} it is proved that 
structure groups of finite non-degenerate involutive solutions are left orderable if and only if they are diffuse. We collect all these facts in the following theorem.

\begin{thm}
	\label{thm:equivalences}
	Let $(X,r)$ be a finite involutive solution. The following statements are equivalent:
	\begin{enumerate}
		\item $(X,r)$ is a multipermutation solution.
		\item $G(X,r)$ is poly-$\Z$.
		\item $G(X,r)$ is left orderable.
		\item $G(X,r)$ is diffuse.
	\end{enumerate}
\end{thm}

As an application of Theorem~\ref{thm:equivalences} we
obtain the following particular case of a theorem proved by Ced\'o, Jespers and Okni\'nski in~\cite{MR2652212}
and by Cameron and Gateva--Ivanova in~\cite{MR2885602}.
For a direct 
proof (without the finiteness assumption), see~\cite[Proposition 10]{Rump}.

\begin{cor}
	\label{cor:cyclic}
	Let $(X,r)$ be a finite involutive solution. If $\mathcal{G}(X,r)$ is
	cyclic, then $(X,r)$ is a multipermutation solution.
\end{cor}

\begin{proof}
	Since $X$ is finite, the group $G(X,r)$ is finitely generated. It is torsion-free and $\Soc(G(X,r))$ is
	an abelian normal subgroup such that
	\[
	G(X,r)/\Soc(G(X,r))\simeq\mathcal{G}(X,r)
	\]
	is cyclic. This implies that
	$G(X,r)$ is left orderable \cite[Lemma 13.3.1]{MR798076} and hence $(X,r)$
	is a multipermutation solution by Theorem~\ref{thm:equivalences}.
\end{proof}

Diffuse groups allow us to obtain
a generalization of Corollary~\ref{cor:cyclic}:

\begin{thm}
	\label{thm:cyclic_sylows}
	Let $(X,r)$ be a finite  involutive solution such that all Sylow subgroups of $\mathcal{G}(X,r)$ are cyclic. Then $(X,r)$ is a multipermutation solution.
\end{thm}

\begin{proof}
    By Theorem~\ref{thm:holonomy}, the structure group $G(X,r)$ is a Bieberbach
	group with holonomy group isomorphic to $\mathcal{G}(X,r)$. Since all Sylow
	subgroups of $\mathcal{G}(X,r)$ are cyclic,	all Bieberbach groups with
	holonomy group isomorphic to $\mathcal{G}(X,r)$ are diffuse
	by~\cite[Theorem 3.5]{MR3548136}. In particular, $G(X,r)$ is diffuse and hence
	the claim follows from Theorem~\ref{thm:equivalences}.
\end{proof}

The converse of Theorem~\ref{thm:cyclic_sylows} does not hold:

\begin{exa}
	Let $X=\{1,2,3,4\}$ and $r(x,y)=(\varphi_x(y),\varphi_y(x))$, where
	\[
		\varphi_1=\varphi_2=\id,\quad
		\varphi_3=(34),\quad
		\varphi_4=(12)(34).
 	\]
	Then $(X,r)$ is an involutive multipermutation solution. One easily checks
	that $\mathcal{G}(X,r)\simeq C_2\times C_2$.
\end{exa}

Let us apply Theorem~\ref{thm:cyclic_sylows} to finite left braces.
The following result of Rump appears in~\cite[Proposition 7]{MR2278047} without the finiteness assumption:
%Cedó, Jespers and Okni\'nski 
%is implicit in~\cite{MR3527540}.

\begin{lem}
    \label{lem:cjo}
    Let $A$ be a finite left brace. Then $(A,r_A)$ is an involutive solution
    such that $\mathcal{G}(A,r_A)\simeq A/\Soc(A)$.
\end{lem}

\begin{proof}
    We only need to prove that $A/\Soc(A)\simeq\mathcal{G}(A,r_A)$. %By~\cite[Theorem 3]{MR3177933},
    The permutation group $\mathcal{G}(A,r_A)=\{\lambda_a:a\in A\}$ is a left brace where the 
    additive structure is given by 
    $\lambda_a+\lambda_b=\lambda_a\lambda_{\lambda^{-1}_a(b)}$ for $a,b\in A$. 
   This implies that the map $\lambda\colon (A,\circ)\to\Aut(A,+)$, $a\mapsto \lambda_a$, 
    is a left brace homomorphism and hence
    \[
    A/\Soc(A)\simeq\lambda(A)=\{\lambda_a:a\in A\}=\mathcal{G}(A,r_A)
    \]
    by the first isomorphism theorem. 
\end{proof}

As an application of Theorem~\ref{thm:cyclic_sylows} we obtain the following
result related to the structure of left braces:

\begin{thm}
	\label{thm:cyclic_sylows:brace}
	Let $A$ be a finite left brace. If all Sylow subgroups of the
	multiplicative group of $A$ are cyclic, then $A$ is right nilpotent. 
\end{thm}

\begin{proof}
    If $(A,\circ)$ has Sylow cyclic subgroups, then $(A/\Soc(A),\circ)$ has
    cyclic Sylow subgroups. By Lemma~\ref{lem:cjo}, 
    $A/\Soc(A)\simeq\mathcal{G}(A,r_A)$ as left braces. In particular, 
    $\mathcal{G}(A,r_A)$ has cyclic Sylow subgroups and
    therefore $(A,r_A)$ is a multipermutation solution 
    by Theorem~\ref{thm:cyclic_sylows}. Now the claim
	follows from~\cite[Proposition 6]{MR3574204}.
\end{proof}

The following consequence of
Theorem~\ref{thm:cyclic_sylows:brace} is immediate:

\begin{cor}
	Let $A$ be a non-trivial finite left brace. 
	If all Sylow subgroups of the
	multiplicative group are cyclic, then $A$ is not simple.
\end{cor}

It is natural to ask whether Theorems~\ref{thm:cyclic_sylows} and~\ref{thm:cyclic_sylows:brace} can be
proved for groups with abelian Sylow subgroups. The following example answers
this question negatively.

\begin{exa}\label{ex:simple_no_A-group}
    There exists a unique simple left brace of size $72$, see~\cite[Remark 4.5]{cjo} and ~\cite[Proposition 4.3]{ksv}. The multiplicative group of this left brace 
	is isomorphic to $\Alt_4\times\Sym_3$ and therefore all of its Sylow
	subgroups are abelian.  Since the socle of this left brace is trivial, the canonical solution to the YBE associated with this left brace is not a multipermutation solution (moreover, it is irretractable).
\end{exa}

%\textcolor{red}{In \cite[Example 5.3]{cjo_abundance}, an example of a finite prime non--simple left brace with multiplicative A--group is given. It is a left brace constructed as a semidirect product of a trivial left brace and an asymmetric product of trivial left braces. Indeed, it is of the form $B= ( T \rtimes_\circ S) \rtimes ( \mathbb{Z}/5) $ of size $2^{11}3^25=92160$. Here $\rtimes_\circ$ denotes the asymmetric product.}
%\textcolor{red}{In \cite{ballester}, nearby properties are proved.}

In Section~\ref{right}, 
using the techniques of~\cite{mbbr} and skew left braces of nilpotent type 
we will generalize the results of this section to non-involutive solutions. 
%The following example from~\cite{MR3320237} shows that
%such generalization is needed:

\begin{exa}
Let $G=\{g_j: j\in \mathbb{Z}/8\mathbb{Z}\}$. The operations
\[
g_i+g_j=g_{i+(-1)^ij},\quad
g_i\circ g_j=g_{i+j}
\]
turns $G$ into a skew left brace with multiplicative group isomorphic to the cyclic group $C_8$ of eight elements and nilpotent (non-abelian) additive group isomorphic to the dihedral group $\D_8$ of eight elements. A direct calculation shows that $G$ is right nilpotent.
\end{exa}

\section{Groups with the unique product property}
\label{UPP}

This section is devoted to study the unique product property in structure
groups of involutive solutions.  Recall that a group $G$ has the \emph{unique product property} if 
for all finite non-empty subsets $A$ and $B$ of $G$ there exists $x\in G$ 
that can be written uniquely as $x = ab$ with $a\in A$ and $b\in B$. 
We refer to~\cite{MR798076} for more information related to the unique product property. 

It is natural to ask when $G(X,r)$ has the unique product property, see~\cite[Section 8]{LV3}. 
If $(X,r)$ is a multipermutation solution, then $G(X,r)$ has the unique product property since $G(X,r)$ is left orderable.

All involutive solutions of size $\leq 8$ were constructed by
Etingof, Schedler and Soloviev in~\cite{MR1722951}. There are $38698$ solutions
and among them only $2583$ are not multipermutation solutions, see Table~\ref{tab:MP}.  Our aim is to
know when the structure group of a not multipermutation involutive 
solution does not have the unique product property. We start
with the following observation made by Jespers and
Okni\'nski:

\begin{pro}
	\label{pro:4-13}
	Let $X=\{1,2,3,4\}$ and $r(x,y)=(\sigma_x(y),\tau_y(x))$ be the irretractable involutive solution given by
	\begin{align*}
		&\sigma_1=(34), && \sigma_2=(1324), && \sigma_3=(1423), && \sigma_4=(12),\\
		&\tau_1=(24), &&\tau_2=(1432), && \tau_3=(1234), && \tau_4=(13).
	\end{align*}
	The structure group $G(X,r)$ with generators
	$x_1,x_2,x_3,x_4$ and relations
	\begin{align*}
		& x_1x_2=x_2x_4,
		&& x_1x_3=x_4x_2,
		&& x_1x_4=x_3^2,\\
		& x_2x_1=x_3x_4,
		&& x_2^2=x_4x_1,
		&& x_3x_1=x_4x_3.
	\end{align*}
	does not have the unique product property.
\end{pro}

\begin{proof}
    See~\cite[Example 8.2.14]{MR2301033}.
\end{proof}

To prove Proposition~\ref{pro:4-13} Jespers and Okni\'nski found
a subgroup of the structure group isomorphic to the Promislow subgroup. 
This idea motivates the results of this section. 

\begin{table}
    \centering
    \caption{The number of (not multipermutation) involutive solutions.}
    \begin{tabular}{|c|c|c|c|c|c|c|c|c|}
        \hline
        $n$ & 1 & 2 & 3 & 4 & 5 & 6 & 7 & 8\\
        \hline
        solutions &  1 & 2 & 5 & 23 & 88 & 595 & 3456 & 34528\\
        not multipermutation & 0 & 0 & 0 & 2 & 4 & 41 & 161 & 2375\\
        \hline
    \end{tabular}
    \label{tab:MP}
\end{table}

\begin{pro}
    \label{pro:4-19}
	Let $X=\{1,2,3,4\}$ and $r(x,y)=(\sigma_x(y),\tau_y(x))$ be the irretractable involutive solution given by 
	\begin{align*}
		&\sigma_1=(12), && \sigma_2=(1324), && \sigma_3=(34), && \sigma_4=(1423),\\
		&\tau_1=(14), &&\tau_2=(1243), && \tau_3=(23), && \tau_4=(1342).
	\end{align*}
	Then the group $G(X,r)$ with 
	generators
	$x_1,x_2,x_3,x_4$ and relations
	\begin{align*}
		& x_1^2=x_2x_4,
		&& x_1x_3=x_3x_1,
		&& x_1x_4=x_4x_3,\\
		& x_2x_1=x_3x_2,
		&& x_2^2=x_4^2,
		&& x_3^2=x_4x_2.
	\end{align*}
	does not have the unique product property. 
\end{pro}

\begin{proof}
    Let $x=x_1x_2^{-1}$ and $y=x_1x_3^{-1}$ and 
    \begin{multline}
    \label{eq:Promislow}
    S=\{ x^2y,
    y^2x,
    xyx^{-1},
    (y^2x)^{-1},
    (xy)^{-2},
    y,
    (xy)^2x,
    (xy)^2,\\
    (xyx)^{-1},
    yxy,
    y^{-1},
    x,
    xyx, 
    x^{-1}
	\}.
    \end{multline}
    To prove that $G(X,r)$ does not have
    the unique product property it is enough to prove that 
    each $s\in S^2=\{s_1s_2:s_1,s_2\in S\}$ admits at least two different decompositions 
    of the form $s=ab=uv$ for $a,b,u,v\in S$. To perform these calculations we 
    use the injective group homomorphism $G\to\GL(5,\Z)$ of Theorem~\ref{thm:ESS}, 
	\begin{align*}
	&x_1\mapsto\left(\begin{smallmatrix}
	0 & 1 & 0 & 0 & 1\\
	1 & 0 & 0 & 0 & 0\\
	0 & 0 & 1 & 0 & 0\\
	0 & 0 & 0 & 1 & 0\\
	0 & 0 & 0 & 0 & 1
  	\end{smallmatrix}\right),
  	&&
	x_2\mapsto\left(\begin{smallmatrix}
	0 & 0 & 0 & 1 & 0\\
	0 & 0 & 1 & 0 & 1\\
	1 & 0 & 0 & 0 & 0\\
	0 & 1 & 0 & 0 & 0\\
	0 & 0 & 0 & 0 & 1
  	\end{smallmatrix}\right),
  	\\
  	&x_3\mapsto\left(\begin{smallmatrix}
	1 & 0 & 0 & 0 & 0\\
	0 & 1 & 0 & 0 & 0\\
	0 & 0 & 0 & 1 & 1\\
	0 & 0 & 1 & 0 & 0\\
	0 & 0 & 0 & 0 & 1
  	\end{smallmatrix}\right),
  	&&
  	x_4\mapsto\left(\begin{smallmatrix}
	0 & 0 & 1 & 0 & 0\\
	0 & 0 & 0 & 1 & 0\\
	0 & 1 & 0 & 0 & 0\\
	1 & 0 & 0 & 0 & 1\\
	0 & 0 & 0 & 0 & 1
  	\end{smallmatrix}\right).
  	\end{align*}
  	This faithful representation of $G(X,r)$ allows us to compute all possible products of the form $s_1s_2$ for all $s_1,s_2\in S$. By inspection, each element of $S^2$ admits at least two different representations.
\end{proof}

\begin{rem}
    The solutions of Propositions~\ref{pro:4-13} and~\ref{pro:4-19}
    are the only two involutive solutions of size four that are not multipermutation solutions. Therefore
    structure groups of involutive solutions of size four that are not multipermutation solutions do not have the unique product
    property. 
\end{rem}

\begin{rem}
    The set~\eqref{eq:Promislow} appears in the work of Promislow~\cite{MR940281}.
\end{rem}

\begin{rem}
    The technique used to prove Proposition~\ref{pro:4-19} could be used to 
    prove Proposition~\ref{pro:4-13}. 
\end{rem}

%With the same methods we could
%that all structure groups of not multipermutation involutive solutions of size $\leq7$ do not have the unique product property. 

\begin{pro}
    \label{pro:size7}
    Let $G(X,r)$ be the structure group of a not multipermutation involutive solution
    of size $\leq7$. Then $G(X,r)$ does not have the unique product property.
\end{pro}

\begin{proof}
    The proof is a case-by-case analysis 
    using the technique used to prove Proposition~\ref{pro:4-19}
    and the list of solutions of size $\leq7$ of~\cite{MR1722951}. 
    In several cases, the elements $x$ and $y$ that realize the set~\eqref{eq:Promislow} were found after a random search. 
\end{proof}

In principle, the argument used to prove Propositions~\ref{pro:4-19},~\ref{pro:size7} and~\ref{pro:GI} could be used for solutions of size eight. The following solution appeared
in~\cite{MR3437282} as a counterexample to a conjecture of Gateva--Ivanova related to the retractability of square-free solutions, see~\cite[2.28(1)]{MR2095675}.

\begin{pro}
    \label{pro:GI}
	Let $X=\{1,\dots,8\}$ and $r(x,y)=(\varphi_x(y),\varphi_y(x))$ be the irretractable involutive solution given by
	\begin{align*}
		& \varphi_1=(78), && \varphi_2=(56), && \varphi_3=(25)(46)(78), && \varphi_4=(17)(38)(56),\\
		& \varphi_5=(24), && \varphi_6=(17)(24)(38), && \varphi_7=(13), && \varphi_8=(13)(25)(46).
	\end{align*}
	Then $G(X,r)$ does not have the unique 
	product property.
\end{pro}

\begin{proof}
    Let $x=x_4x_2^{-1}x_1x_3^{-1}$ and
	$y=x_1x_2^{-1}x_3x_1^{-1}x_4x_1^{-1}$. (These elements were found after a random search.) 
	The injective
	group homomorphism $G(X,r)\to\GL(9,\Z)$ of Theorem~\ref{thm:ESS} 
	allows us to use the set~\eqref{eq:Promislow} to
	prove that $G(X,r)$ does not have the unique
	product property. 
\end{proof}

There are solutions of size eight where our technique 
does not seem to work. One of these solutions appears in the following example:

\begin{exa}
\label{exa:15579}
Let $X=\{1,\dots,8\}$ and $r(x,y)=(\sigma_x(y),\tau_y(x))$, where 
\begin{align*}
&\sigma_1=\sigma_2=(3745), && \tau_1=\tau_2=(3648),\\
&\sigma_3=\sigma_4=(1826), && \tau_3=\tau_4=(1527),\\
&\sigma_5=\sigma_7=(13872465), && \tau_5=\tau_7=(16542873),\\
&\sigma_6=\sigma_8=(17842563), && \tau_6=\tau_8=(13562478).
\end{align*}
Then $(X,r)$ is an involutive solution 
that retracts to the solution of Proposition~\ref{pro:4-13}. In particular, $(X,r)$ is not a multipermutation solution. 
\end{exa}

Table~\ref{tab:4-13} shows four involutive solutions that retract 
to the solution of Proposition~\ref{pro:4-13} and where 
our technique does not seem to work; the solution of Example~\ref{exa:15579} is the first entry of Table~\ref{tab:4-13}. In Table~\ref{tab:4-19} one finds four involutive solutions
that retract to the solution of Proposition~\ref{pro:4-19} and where our technique 
does not seem to work. We do not know whether the structure groups of the solutions 
of Tables~\ref{tab:4-13} and~\ref{tab:4-19}
have the unique product property. 
% In Section~\ref{Promislow} we will prove that the 
% structure groups of those solutions of Tables~\ref{tab:4-13} and~\ref{tab:4-19} 
% do not contain a subgroup isomorphic to the Promislow subgroup; this explains why
% the search of sets like~\eqref{eq:Promislow} is hopeless.

\begin{table}
\caption{Some solutions that retract to the solution of Proposition~\ref{pro:4-13}.}
\begin{center}
\begin{tabular}{|c|c|c||c|c|} 
\hline
$x$ & $\sigma_x$ & $\tau_x$ & $\sigma_x$ & $\tau_x$\\
\hline
1 & $(3745)$ & $(3648)$ & $(3745)(68)$ & $(3648)(57)$\\
2 & $(3745)$ & $(3648)$ & $(3745)(68)$ & $(3648)(57)$\\
3 & $(1826)$ & $(1527)$ & $(1826)(57)$ & $(1527)(68)$\\
4 & $(1826)$ & $(1527)$ & $(1826)(57)$ & $(1527)(68)$\\
5 & $(13872465)$ & $(16542873)$ & $(1465)(2387)$ & $(1654)(2873)$\\
6 & $(17842563)$ & $(13562478)$ & $(1784)(2563)$ & $(1478)(2356)$\\
7 & $(13872465)$ & $(16542873)$ & $(1465)(2387)$ & $(1654)(2873)$\\
8 & $(17842563)$ & $(13562478)$ & $(1784)(2563)$ & $(1478)(2356)$\\
\hline
1 & $(12)(4675)$ & $(12)(3685)$ & $(12)(35)(4867)$ & $(12)(3857)(46)$\\ 
2 & $(12)(4675)$ & $(12)(3685)$ & $(12)(35)(4867)$ & $(12)(3857)(46)$\\
3 & $(1435)(2786)$ & $(1578)(2643)$ & $(16582437)$ & $(17652843)$\\
4 & $(1587)(2634)$ & $(1345)(2876)$ & $(17562834)$ & $(15682347)$\\
5 & $(1823)(56)$ & $(1724)(56)$ & $(16582437)$ & $(17652843)$\\
6 & $(1823)(56)$ & $(1724)(56)$ & $(17562834)$ & $(15682347)$\\
7 & $(1587)(2634)$ & $(1345)(2876)$ & $(1325)(46)(78)$ & $(1426)(35)(78)$\\
8 & $(1435)(2786)$ & $(1578)(2643)$ & $(1325)(46)(78)$ & $(1426)(35)(78)$\\
\hline
\end{tabular}
\end{center}
\label{tab:4-13}
\end{table}

\begin{table}
\caption{Some solutions that retract to the solution of Proposition~\ref{pro:4-19}.}
\begin{center}
\begin{tabular}{|c|c|c||c|c|} 
\hline
$x$ & $\sigma_x$ & $\tau_x$ & $\sigma_x$ & $\tau_x$\\
\hline
1 & $(12)(78)$ & $(14)(67)$ & $(12)(35)(46)(78)$ & $(14)(28)(35)(67)$\\
2 & $(1584)(2673)$ & $(1265)(3784)$ & $(1324)(5867)$ & $(1243)(5786)$\\
3 & $(34)(56)$ & $(23)(58)$ & $(17)(28)(34)(56)$ & $(17)(23)(46)(58)$\\
4 & $(1485)(2376)$ & $(1562)(3487)$ & $(1423)(5768)$ & $(1342)(5687)$\\
5 & $(34)(56)$ & $(23)(58)$ & $(17)(28)(34)(56)$ & $(17)(23)(46)(58)$\\
6 & $(1485)(2376)$ & $(1562)(3487)$ & $(1423)(5768)$ & $(1342)(5687)$\\ 
7 & $(12)(78)$ & $(14)(67)$ & $(12)(35)(46)(78)$ & $(14)(28)(35)(67)$\\
8 & $(1584)(2673)$ & $(1265)(3784)$ & $(1324)(5867)$ & $(1243)(5786)$\\
\hline
1 & $(13687542)$ & $(13867524)$& $(1652)$ & $(1854)$\\
2 & $(17)(2583)(46)$ & $(1278)(35)(46)$ & $(17645328)$ & $(12835647)$\\
3 & $(18657243)$ & $(16857423)$ & $(3874)$ & $(2367)$\\
4 & $(1476)(28)(35)$ & $(17)(28)(3456)$ & $(14635827)$ & $(17825346)$\\
5 & $(18657243)$ & $(16857423)$ & $(1652)$ & $(1854)$\\
6 & $(1476)(28)(35)$ & $(17)(28)(3456)$ & $(17645328)$ & $(12835647)$\\
7 & $(13687542)$ & $(13867524)$ & $(3874)$ & $(2367)$\\
8 & $(17)(2583)(46)$ & $(1278)(35)(46)$ & $(14635827)$ & $(17825346)$\\
\hline
\end{tabular}
\end{center}
\label{tab:4-19}
\end{table}
%\end{multicols}

\section{Finding Promislow subgroups}
\label{Promislow}

In this section we explain the general theory we will use to find subgroups isomorphic to the Promislow group in a given Bieberbach group. The Promislow group was the first example of a torsion-free group that does not have  the unique product property, see~\cite{MR940281}.

\begin{lem}
    Let $P$ be the Promislow group
        \[
        \langle x,y \mid x^{-1}y^2x=y^{-2}, y^{-1}x^2y=x^{-2} \rangle.
        \]
    Then $A=\langle x^2, y^2, (xy)^2\rangle$ is a normal free abelian subgroup of $P$ of rank 3 with $P/A$ isomorphic to the Klein group. Furthermore, $P$ is torsion-free and not left orderable. 
\end{lem}

\begin{proof}
    See for example~\cite[Lemma 13.3.3]{MR798076}.
\end{proof}

Let $\Gamma\subseteq\GL(n,\Z) \ltimes \Z^n$ be a Bieberbach group defined by the following short exact sequence
\begin{equation}\label{eq:Bieberbach}
    0 \longrightarrow L \longrightarrow \Gamma \overset{\pi}{\longrightarrow} \Gamma/L \longrightarrow 0.
\end{equation}
Here $L \subseteq \Z^n$ is taken such that $\{ I \} \times L$ is the maximal normal abelian subgroup of $\Gamma$, where $I$ denotes the identity matrix in $\GL(n,\Z)$ and $\pi$
is the canonical map, i.e. $\pi(A, a)=A$.

We say that elements $x,y$ of a group $G$ satisfy~\eqref{eq:Promislow_relation} if and only if
\begin{equation}
\label{eq:Promislow_relation}
    x^2y=yx^{-2} \quad \text{and} \quad y^2x=xy^{-2} \tag{P}
\end{equation}
holds in $G$.

\begin{lem}
    \label{lem:Promislow_relation_mat}
    Let $\alpha=(A,a)$ and $\beta=(B,b)$ be elements of $\Gamma$ that generate a subgroup isomorphic to $P$. 
    Then the following statements hold:
    \begin{enumerate}
        \item $A\neq I$ and $B\neq I$.
        \item $A$ and $B$ satisfy~\eqref{eq:Promislow_relation}. 
        %$A^2B=BA^{-2}$ and $B^2A=AB^{-2}$.
    \end{enumerate}
\end{lem}

\begin{proof}
    We have the following short exact sequence:
    \[
        0 \longrightarrow \langle \alpha^2,\beta^2,(\alpha\beta)^2 \rangle \longrightarrow P \longrightarrow C_2^2 \longrightarrow 0.
    \]

    To prove that $A\neq I$ and $B\neq I$, let us assume that $A=I$. Let $k\in\mathbb{N}$ be such that $\beta^{2k}=(I,b')$; this is possible because $\pi(\beta)=B$ lies on a finite group. Then 
    \[
    \beta^{-2k}=\alpha^{-1}\beta^{2k}\alpha=(I,-a)(I,b')(I,a)=(I,b')=\beta^{2k},
    \]
    a contradiction since $\Gamma$ is torsion free.
    
    To prove that $A$ and $B$ satisfy~\eqref{eq:Promislow_relation} 
    just notice that $\pi(\alpha)=A$ and $\pi(\beta)=B$ and that $\alpha,\beta$ satisfy~\eqref{eq:Promislow_relation}.
\end{proof}

We will make use of two Laurent polynomials
\begin{align*}
    P_1(X,Y)=1+X+YX^{-1}+YX^{-2},&&
    P_2(X,Y)=-1+X^2.
\end{align*}

\begin{lem}
    \label{lem:integral_sol}
    Let $G$ be a group and $A,B\in G$ be two elements that satisfy~\eqref{eq:Promislow_relation}. Let $(A,v)$ and $(B,w)$ be any pair of elements of $\Gamma$ that projects to $A$ and $B$ respectively. If 
    \[
    \left[ {\begin{array}{cc}
         P_1(A,B) & P_2(A,B)  \\
         P_2(B,A) & P_1(B,A)
    \end{array}} \right]
    \left[ {\begin{array}{c}
         x  \\
         y
    \end{array}} \right]=
    -\left[ {\begin{array}{cc}
         P_1(A,B) & P_2(A,B)  \\
         P_2(B,A) & P_1(B,A)
    \end{array}} \right]
     \left[ {\begin{array}{c}
         v  \\
         w
    \end{array}} \right]
    \]
    has an integral solution $x,y\in L$, then
        \[
            \alpha=(A,x+v) \quad \text{and} \quad \beta=(B,y+w)
        \]
    satisfy~\eqref{eq:Promislow_relation}.
    
    \begin{proof}
        We prove that $\alpha^2\beta=\beta\alpha^{-2}$. 
        By %Lemma~\ref{lem:Promislow_relation_mat}, 
        assumption, $A^2B=BA^{-2}$. Then, 
        using the identification of $\alpha$ and $\beta$ as matrices,
        we see that $\alpha^2\beta=\beta\alpha^{-2}$ is equivalent to 
        $P_1(A,B)(x+v)+P_2(A,B)(y+w)=0$,
        which is true by hypothesis. Similarly one proves that $\beta^2\alpha=\alpha\beta^{-2}$. 
    \end{proof}
\end{lem}

\begin{pro}
    \label{prop:existence_promislow}
    Let $\Gamma$ be a group defined by a short exact sequence as \eqref{eq:Bieberbach}. Let $\alpha,\beta\in\Gamma$ be such that $\pi(\alpha)\neq I$, $\pi(\beta)\neq I$. If $\alpha$ and $\beta$ satisfy~\eqref{eq:Promislow_relation}, then they generate a subgroup of $\Gamma$ isomorphic $P$.
\end{pro}

\begin{proof}
    Let $P=\langle\alpha,\beta\rangle$, $L_P=\langle a,b,c\rangle$ where $a=\alpha^2, b=\beta^2, c=(\alpha\beta)^2$. Then $P$ is a Bieberbach group which fits into the short exact sequence
        \[
    0 \longrightarrow L_P \longrightarrow P \longrightarrow C_2^2 \longrightarrow 0.
    \]
    $L_P$ is an abelian subgroup of $\Gamma$, hence it is free abelian and it is maximal normal abelian subgroup of $P$. It is enough to show that $L_P$ is of rank 3. Let $n_a,n_b,n_c$ be integers such that $ a^{n_a}b^{n_b}c^{n_c}=1$. 
    Conjugation by $\alpha$ leaves 
    $a^{n_a}b^{-n_b}c^{-n_c}=1=a^{n_a}b^{n_b}c^{n_c}$ 
    and 
    hence
    $
    b^{2n_b}c^{2n_c}=1
    $.
    Now, conjugation by $\beta$ gives us $c^{4n_c}=1$. Since $\Gamma$ is a torsion free group, we conclude that  $n_a=n_b=n_c=0$.
\end{proof}
    
\begin{rem}
    Calculations of the previous proposition are easily checked using the representation from~\cite[Lemma 1]{MR3723179} that we state here for completeness:
\[ 
\left\langle 
\alpha=\left(\begin{smallmatrix}
     1 & 0 & 0 & 1/2  \\
     0 & -1 & 0 & 1/2 \\
     0 & 0 & -1 & 0 \\
     0 & 0 & 0 & 1
\end{smallmatrix}\right),
\beta=\left(\begin{smallmatrix}
     -1 & 0 & 0 & 0  \\
     0 & 1 & 0 & 1/2 \\
     0 & 0 & -1 & 1/2 \\
     0 & 0 & 0 & 1
\end{smallmatrix}\right)\right\rangle\subseteq\GL(4,\Q).
\]
\end{rem}

We now present an algorithm for finding subgroups (of a Bieberbach group) 
that are isomorphic to the Promislow group:

\begin{algorithm}
    \label{alg:Promislow}
    Let $\Gamma\subseteq\GL(n,\Z) \ltimes \Z^n$ be a Bieberbach group defined by the following short exact sequence
    \[
    0 \longrightarrow L \longrightarrow \Gamma \overset{\pi}{\longrightarrow} \Gamma/L \longrightarrow 0,
    \]
    where $L \subseteq \Z^n$ is taken such that $\{ I \} \times L$ is the maximal normal abelian subgroup of $\Gamma$ and $\pi$ is the canonical map.
    \begin{enumerate}
        \item Check all pairs $A,B\in G\setminus\{1\}$ that satisfy~\eqref{eq:Promislow_relation}.
        \item Determine preimages $(A,v)\in\pi^{-1}(A)$ and $(B,w)\in\pi^{-1}(B)$.
        \item Check if the linear system of Lemma~\ref{lem:integral_sol} has integer solutions. By  Proposition~\ref{prop:existence_promislow}, the existence of such solutions is equivalent to 
        the existence of a subgroup isomorphic to $P$.
    \end{enumerate}
\end{algorithm}

As an application, we obtain the following improvement of Proposition~\ref{pro:size7}:

\begin{thm}
\label{thm:P}
    Let $G(X,r)$ be the structure group of a not multipermutation involutive solution
    of size $\leq8$. Then $G(X,r)$ contains a subgroup isomorphic to the Promislow subgroup if and only if
    $(X,r)$ is not isomorphic to the solutions of Tables~\ref{tab:4-13} and~\ref{tab:4-19}. 
\end{thm}

\begin{proof}
    The proof is a case-by-case analysis 
    using Algorithm~\ref{alg:Promislow} and the list of involutive solutions of~\cite{MR1722951}.
\end{proof}    

% The~\textsf{GAP} script used to prove Theorem~\ref{thm:P} is available immediately from
% the authors on request.

\section{Right $p$-nilpotent skew left braces}
\label{right}

Let $A$ be a skew left brace. For subsets $X$ and $Y$ of $A$ we define inductively $R_0(X,Y)=X$ and $R_{n+1}(X,Y)$ as the additive subgroup generated by
$R_n(X,Y)*Y$ and $[R_n(X,Y),Y]_+$ for $n\geq0$.

\begin{lem}
    \label{lem:R:inclusion}
    Let $I$ be an ideal of a skew left brace $A$. Then
    $R_{n+1}(I,A)\subseteq R_n(I,A)$ for all $n\geq0$.
\end{lem}

\begin{proof}
    We proceed by induction on $n$. The case $n=0$ is trivial as $I$ is an ideal of $A$. Let us assume that the claim holds for some $n\geq0$. Since by the inductive hypothesis
    $R_n(I,A)*A\subseteq R_{n-1}(I,A)*A\subseteq R_{n}(I,A)$ and 
    \[
    [R_n(I,A),A]_+\subseteq [R_{n-1}(I,A),A]_+ \subseteq R_n(I,A),
    \]
    it follows that $R_{n+1}(I,A)\subseteq R_n(I,A)$.
\end{proof}

\begin{pro}
    \label{lem:R:ideal}
    Let $I$ be an ideal of a skew left brace $A$. Then each $R_{n}(I,A)$ is an ideal of $A$. %contained in $R_n(I,A)$. %and $R_{n+1}(I,A)\subseteq R_n(I,A)$ for all $n\geq0$.
\end{pro}

\begin{proof}
    We proceed by induction on $n$. The case where $n=0$ follows from the fact that $I$ is an ideal of $A$. 
    So assume that the result holds for some $n\geq0$.
    We first prove that $R_{n+1}(I,A)$ is a normal subgroup of $(A,+)$. Let $a,b\in A$ and 
    $x\in R_n(I,A)$. Then
    \[
    a+x*b-a=-x*a+x*(a+b)\in R_{n+1}(I,A),
    \]
    by definition. 
    Since moreover
    \[
    a+(x+b-x-b)-a=(a+x-a)+(a+b-a)-(a+x-a)-(a+b-a) \in R_{n+1}(I,A)
    \]
    by the inductive hypothesis, it follows that $R_{n+1}(I,A)$ is a normal subgroup of $(A,+)$.
  
    We now prove that 
    \begin{equation}
        \label{eq:lambda}\lambda_a(R_{n+1}(I,A))\subseteq R_{n+1}(I,A)
    \end{equation} 
    for all $a\in A$. 
    Using 
    the inductive hypothesis and that 
    each $\lambda_a\in\Aut(A,+)$, 
    \begin{align*}
  %  &\lambda_a([R_n(I,A),A]_+)\subseteq[\lambda_a(R_n(I,A)),\lambda_a(A)]_+
  %  \subseteq [R_n(I,A),A]_+
  %  \subseteq R_{n+1}(I,A),
    &\lambda_a(x*b)%&=\lambda_a(\lambda_x(b)-b) = \lambda_{axa'}\lambda_a(b)-\lambda_a(b)\\
    =(a\circ x\circ a')*\lambda_a(b) \in R_{n+1}(I,B)
    \shortintertext{and}
    &\lambda_a([R_n(I,A),A]_+)\subseteq[\lambda_a(R_n(I,A)),\lambda_a(A)]_+\\
    &\phantom{\lambda_a([R_n(I,A),A]_+)}\subseteq [R_n(I,A),A]_+
    \subseteq R_{n+1}(I,A),
    % &\deleted{\lambda_a(R_n(I,A)*A)\subseteq \lambda_a(R_n(I,A))*\lambda_a(A)
    % \subseteq R_n(I,A)*A\subseteq R_{n+1}(I,A),}
    \end{align*}
    equality \eqref{eq:lambda} follows. 
    
    Since $R_{n+1}(I,A)\subseteq R_n(I,A)$ by Lemma \ref{lem:R:inclusion},
    \[
    R_{n+1}(I,A)*A\subseteq R_n(I,A)*A\subseteq R_{n+1}(I,A).
    \] 
    Hence the claim follows from~\cite[Lemma 1.9]{csv}.
    %Since both $R_n(I,A)*A\subseteq R_n(I,A)$ and 
    %$[R_n(I,A),A]_+\subseteq R_n(I,A)$
    %hold by the inductive hypothesis, 
    %$R_{n+1}(I,A)*A\subseteq R_n(I,A)*A\subseteq R_{n+1}(I,A)$ by using %Lemma~\ref{lem:R:inclusion}.
    %Hence the claim follows from~\cite[Lemma 1.9]{csv}.
\end{proof}

\begin{lem}\label{lem:mbbr}
	Let $A$ be a skew left brace, $X$ be a subset of $A$ and $n,m\in\N$.
	Then $R_m(X,A)\subseteq \Soc_n(A)$ if and only if $X\subseteq\Soc_{m+n}(A)$.
\end{lem}

\begin{proof}
    We proceed by induction on $m$. The case where $m=0$ is trivial, so assume that the result is valid for some $m\geq0$. 
	Note that $R_{m+1}(X,A)\subseteq\Soc_{n}(A)$ is equivalent to  $R_m(X,A)*A\subseteq\Soc_{n}(A)$ and $[R_m(X,A),A]_+ \subseteq\Soc_{n}(A)$. By Lemma~\ref{lem:soc_n}, this is equivalent to $R_m(X,A)\subseteq\Soc_{n+1}(A)$, which is equivalent to $X\subseteq\Soc_{m+n+1}(A)$ by the inductive hypothesis.
\end{proof}

\begin{lem}
	A skew left brace $A$ of nilpotent type is right nilpotent if and only if $R_n(A,A)=0$ for some $n\in\N$.
%	
%	\begin{proof}
%		If $A$ is right nilpotent, Lemma \ref{lem:soc_n_2} implies that $A=R_0(A,A)\subseteq \Soc_n(A)$ for some $n$ and hence 
%		$R_n(A,A)=0$ by Lemma \ref{lem:mbbr}. 
%		Conversely, if $R_n(A,A)=0$, we prove by induction that $A^{(n+1)}\subseteq R_n(A,A)$ and then $A^{(n+1)}=0$.
%	\end{proof}
\end{lem}

\begin{proof}
    By Lemma \ref{lem:mbbr}, $R_n(A,A)=0$ if and only if $A=\Soc_{n}(A)$. By Lemma \ref{lem:soc_n_2}, the latter is equivalent to $A$ being right nilpotent.
\end{proof}

Recall that a finite group $G$ is said to be \emph{$p$-nilpotent} if there exists 
a normal Hall $p'$-subgroup of $G$. One proves that this subgroup is characteristic in $G$. 
Following~\cite{mbbr} we define right $p$-nilpotent skew left braces
of nilpotent type:

\begin{defn}
    Let $p$ be a prime number. A finite skew left brace $A$ of nilpotent type is said to be \emph{right $p$-nilpotent} if there exists $n\geq1$ such that $R_n(A_p,A)=0$, where $A_p$ is the Sylow $p$-subgroup of $(A,+)$.
\end{defn}

\begin{pro}\label{prop:soc_n}
    Let $A$ be a finite skew left brace of nilpotent type and $p\in\pi(A)$. Then
    $A_p\subseteq\Soc_n(A)$ for some $n\geq1$ if and only if 
    $A$ is right $p$-nilpotent. 
\end{pro}

\begin{proof}
    By Lemma~\ref{lem:mbbr}, %\textcolor{red}{or Lemma \ref{lem:soc_n_3}}, 
    $R_n(A_p,A)=0$ if and only if $A_p\subseteq\Soc_n(A)$. 
\end{proof}

\begin{pro}
    A finite skew left brace $A$ of nilpotent type is right nilpotent if and only if $A$ is right $p$-nilpotent for all $p\in\pi(A)$. 
\end{pro}

\begin{proof}
    Assume first that $A$ is right nilpotent. 
    By Lemma~\ref{lem:soc_n_2}, there exists $n\in\N$ such that $A_p\subseteq A=\Soc_n(A)$ for all $p\in\pi(A)$. Hence the claim follows from
    Proposition~\ref{prop:soc_n}. 
    Assume now
    that $A$ is right $p$-nilpotent for all $p\in\pi(A)$. This means
    that for each $p\in\pi(A)$ there exists $n(p)\in\N$ such that
    $A_p\subseteq\Soc_{n(p)}(A)$. Let 
    $n=\max\{n(p):p\in\pi(A)\}$. Then $A_p\subseteq\Soc_n(A)$ for all $p\in\pi(A)$. Since $\Soc_n(A)$ is an ideal
    of $A$ and $A$ is of nilpotent type, $A=\oplus_{p\in\pi(A)}A_p\subseteq\Soc_n(A)$. Hence $A$ is right nilpotent by Lemma~\ref{lem:soc_n_2}.
\end{proof}

%\textcolor{red}{\begin{rem}
%	If $X\subseteq Y$, then $R_n(X,a)\subseteq(Y,A)$. It follows easily by induction and the first part of the previous proposition is straightforward.
%\end{rem}}

In~\cite{mbbr}, Meng, Ballester--Bolinches and Romero prove the following theorem 
for left braces: 
%The main result of this section is the following:

\begin{thm}
\label{thm:right_p}
    Let $A$ be a finite skew left brace of nilpotent type. If $(A,\circ)$ has an abelian normal
    Sylow $p$-subgroup for some $p\in\pi(A)$, then $A$ is right $p$-nilpotent.
\end{thm}

Our proof is very similar to that of~\cite{mbbr}. 
We shall need the following lemmas:

\begin{lem}
\label{lem:A_p:ideal}
    Let $A$ be a finite skew left brace of nilpotent type. If $(A,\circ)$ has a normal
    Sylow $p$-subgroup for some $p\in\pi(A)$, then $A_p$ is an ideal of $A$.
\end{lem}

\begin{proof}
   Since the group $(A,+)$ is nilpotent, there exists a unique normal Sylow $p$-subgroup $A_p$ of $(A,+)$. 
   By Lemma~\ref{lem:Sylow:left_ideal}, $A_p$ is a left ideal of $A$. Then $A_p$ is a Sylow
   $p$-subgroup of $(A,\circ)$, normal by hypothesis and hence $A_p$ is an ideal of $A$. 
\end{proof}

\begin{lem}
    \label{lem:Soc(A_p)}
    Let $A$ be a finite skew left brace of nilpotent type.
    If $(A,\circ)$ has a normal Sylow $p$-subgroup for some $p\in\pi(A)$, 
    then $\Soc(A_p)=\Soc(A)\cap A_p$. In particular, $\Soc(A_p)$ is an ideal of $A$.
\end{lem}

\begin{proof}
    By Lemma~\ref{lem:A_p:ideal}, $A_p$ is an ideal of $A$. 
    Clearly $\Soc(A_p)\supseteq\Soc(A)\cap A_p$, so 
    we only need to prove that $\Soc(A_p)\subseteq\Soc(A)\cap A_p$. 
    %If $a\in A_p\cap\Soc(A)$, then 
    %$a\in Z(A_p,+)$ and $a*b=0$ for all $b\in A$. In particular, $a\in Z(A,+)$ and 
    %$a*b=0$ for all $b\in A_p$. Conversely, 
    If $a\in\Soc(A_p)$, then $a\in Z(A_p,+)$ and
    $a*b=0$ for all $b\in A_p$. Let $c\in A$ and write $c=x+y$, where $x\in A_p$ and 
    $y\in A_{p'}$. Since 
    \[
    a*c=a*(x+y)=a*x+x+a*y-x=x+a*y-x\in A_p\cap A_{p'}=0
    \]
    and $a\in Z(A,+)$, the lemma is proved.
\end{proof}

Now we prove Theorem~\ref{thm:right_p}.

\begin{proof}
    Let us assume that the result does not hold and let $A$ be a counterexample of minimal size. We may assume that $A$ is non-trivial, i.e. $\Soc(A)\ne A$.
    %By Lemma~\ref{lem:Soc(A_p)}, $\Soc(A_p)$ is an ideal of $A$. 
    %Let $A_{p'}=\oplus_{q\in\pi(A)\setminus\{p\}}A_q$. The facts that $A_{p'}$ is a left ideal of $A$ 
    %and $A_p$ is an ideal of $A$ imply that $A_p*A_{p'}\subseteq A_p\cap A_{p'}=\{0\}$.
    By Lemma~\ref{lem:A_p:ideal}, $A_p$ is an ideal of $A$.
    
    Since 
    $\lambda_a\in\Aut(A_p,+)$, $\lambda_a(Z(A_p,+))\subseteq Z(A_p,+)$ 
    and hence $Z(A_p,+)$ is a left ideal of $A_p$.
    
    By Lemma~\ref{lem:Soc(A_p)}, $\Soc(A_p)$ is an ideal of $A$. Furthermore, since $(A_p,\circ)$ is abelian, 
    \begin{align*}
    \Soc(A_p)&=\{a\in A_p:a*b=0\text{ for all $b\in A_p$}\}\cap Z(A_p,+)\\
    &=\{a\in A_p:a\circ b=a+b\text{ for all $b\in A_p$}\}\cap Z(A_p,+)\\
    &=\{a\in A_p:b\circ a=b+a\text{ for all $b\in A_p$}\}\cap Z(A_p,+)\\
    &=\Fix(A_p)\cap Z(A_p,+).
    \end{align*}
    Since $|A_p|=p^m$ for some $m\geq1$, the skew left brace $A_p$ is left nilpotent by~\cite[Proposition 4.4]{csv} and, moreover, $Z(A_p,+)$ is a non-zero 
    subgroup of $(A_p,+)$. Then $\Soc(A_p)=\Fix(A_p)\cap Z(A_p,+)\ne 0$ by~\cite[Proposition 2.26]{csv}. In particular, $0\ne \Soc(A_p)\subseteq \Soc(A)$. By Lemma \ref{lem:Soc(A_p)}, $I=\Soc(A_p)$ is a non-trivial ideal of $A$. Then $A/I$ is a skew left brace of nilpotent type such that $0<|A/I|<|A|$. The minimality of $|A|$ implies that $A/I$ is right $p$-nilpotent. Hence, $R_n(A_p/I,A/I)=0$ for some $n$. That is $R_n(A_p,A)\subseteq I\subseteq \Soc(A)$. Now, by Lemma \ref{lem:mbbr}, $R_{n+1}(A_p,A)=0$. Then $A$ is right $p$-nilpotent, a contradiction.
    % \deleted{and hence $A/\Soc(A)$ is skew left brace of nilpotent type 
    % such that $0<|A/\Soc(A)|<|A|$. The minimality of $|A|$ implies
    % that $A/\Soc(A)$ is right nilpotent. Then $A$ is right nilpotent by
    % \cite[Proposition 2.17]{csv}, a contradiction.} 
\end{proof}

Recall that a
group $G$ has the \emph{Sylow tower property} if 
there exists a normal series
$1=G_0\subseteq G_1\subseteq\cdots\subseteq G_n=G$ 
such that each quotient $G_i/G_{i-1}$ is isomorphic to a Sylow subgroup of $G$. 
We also recall that \emph{$A$--groups} are finite groups whose Sylow subgroups are abelian. 

\begin{cor}
\label{cor:STP+abelian}
	Let $A$ be a finite skew left brace of nilpotent type. Assume that 
	$(A,\circ)$ has the Sylow tower property and that 
	all Sylow subgroups of $(A,\circ)$ are abelian. 
	Then $A$ is right nilpotent.
\end{cor}

\begin{proof}
    Assume that the result is not true and let $A$ be a counterexample of minimal size. 
    Since $(A,\circ)$ has the Sylow tower property, there exists a normal Sylow $p$-subgroup $A_p$ of $(A,\circ)$. Then $A_p$ is a non-zero ideal of $A$
    and one proves that 
    \[
    0\ne\Soc(A_p)=\Soc(A)\cap A_p\subseteq\Soc(A).
    \]
    The 
    group $(A/\Soc(A),\circ)$ has abelian Sylow subgroups 
    and has the Sylow tower property. Since $A$ is a non-trivial
    skew left brace, $0<|A/\Soc(A)|<|A|$, and therefore $A/\Soc(A)$ 
    is right nilpotent by the minimality of $|A|$. 
    By~\cite[Proposition 2.17]{csv}, $A$ is right nilpotent, a contradiction.
\end{proof}

There are examples of right nilpotent left braces where the multiplicative 
group contains a non-abelian Sylow subgroup or does not have the Sylow tower property:

\begin{exa}
    The operation $a\circ b=a+3^ab$ turns $\Z/8$ into a right nilpotent left brace with multiplicative group isomorphic to the quaternion group. This example appears in~\cite{MR3320237}.
\end{exa}

\begin{exa}
    Let $G=\Alt_4\times\Sym_3$. Each Sylow subgroups of $G$ is abelian, so it follows from~\cite[Theorem 2.1]{cjo} that there exists a left brace with multiplicative group isomorphic to $G$. The group $G$ does not have the Sylow tower property. The database of left braces of~\cite{MR3647970} shows that there are only four left braces with multiplicative group isomorphic to $G$, all with additive group isomorphic to $C_6\times C_6\times C_2$. However, only one of these four braces is not right nilpotent. 
\end{exa}

As a corollary, we obtain a generalization of Theorem~\ref{thm:cyclic_sylows:brace}:

\begin{cor}
    Let $A$ be a finite skew left brace of nilpotent type. If all Sylow subgroups of the multiplicative
    group of $A$ are cyclic, then $A$ is right nilpotent.
\end{cor}

\begin{proof}
    Since all Sylow subgroups of $(A,\circ)$ are cyclic, the group $(A,\circ)$ is supersolvable and hence
    it has the Sylow tower property. Then the claim follows from Corollary~\ref{cor:STP+abelian}.
\end{proof}

%The same technique can be used to prove a converse of Theorem~\ref{thm:right_p}. 
%\begin{thm}
%    Let $A$ be a finite skew left brace of nilpotent type. If $A$ is right nilpotent 
%    and all Sylow subgroups of $(A,\circ)$ are abelian, then $(A,\circ)$
%    satisfies the Sylow tower property.
%\end{thm}
%\begin{proof}
%    Let $A$ be a counterexample of minimal size. By Lemma~\ref{lem:A_p:ideal}, $A_p$ is an ideal of $A$ for 
%    some $p\in\pi(A)$. Since $A_p$ satisfies the hypothesis of the theorem, the minimality of $|A|$ implies that %$(A_p,\circ)$ has the Sylow tower property. From this the claim follows. 
%\end{proof}

\section{Left $p$-nilpotent skew left braces}
\label{left}

Let $A$ be a skew left brace. For subsets $X$ and $Y$ of $A$ 
we define inductively $L_0(X,Y)=Y$ and $L_{n+1}(X,Y)=X*L_n(X,Y)$ for $n\geq0$.

\begin{defn}
    Let $p$ be a prime number. A finite skew left brace $A$ of nilpotent type is said to be \emph{left $p$-nilpotent} if there exists $n\geq1$ such that $L_n(A,A_p)=0$, where $A_p$ is the Sylow $p$-subgroup of $(A,+)$.
\end{defn}

\begin{lem}
\label{lem:factorization}
    Let $A$ be a skew left brace such that its additive group is the direct product of the left ideals $B$ and $C$. Then $A*(B+C)=A*B+A*C$. Moreover, if $A=\oplus_{i=1}^{n} B_i$ where the $B_i$ are left ideals, then 
    \[
    A*\sum_{i=1}^{n} B_i=\sum_{i=1}^{n} A*B_i.
    \]
\end{lem}

\begin{proof}
    Let $a\in A$, $b\in B$ and $c\in C$. 
    Then 
    \[
    a*(b+c)=a*b+b+a*c-b=a*b+a*c
    \]
    holds for all $a\in A$, $b\in B$ and $c\in C$.
    The second part follows by induction.
\end{proof}

\begin{pro}
    \label{pro:left_p}
    Let $A$ be a finite skew left brace of nilpotent type. Then $A$ 
    is left nilpotent if and only if $A$ is left $p$-nilpotent for all $p\in\pi(A)$.
\end{pro}

\begin{proof}
    For each $p\in\pi(A)$ 
    there exists $n(p)\in\N$ such that $L_{n(p)}(A,A_p)=0$. Let
    $n=\max\{n(p):p\in\pi(A)\}$. Then $L_n(A,A_p)=0$ for all $p\in\pi(A)$. Since
    $A$ is of nilpotent type, the group $(A,+)$ is isomorphic to the direct sum 
    of the $A_p$ for $p\in\pi(A)$. 
    Then Lemma~\ref{lem:factorization} implies that 
    \[
    L_n(A,A)=\sum_{p\in\pi(A)}L_n(A,A_p)=0.
    \]
    The other implication is trivial. 
\end{proof}

We now recall some notation about commutators. Given a skew left brace $A$, the group $(A,\circ)$ acts on $(A,+)$ by automorphisms. If in the semidirect product $(A,+)\rtimes(A,\circ)$ we identify $a$ with $(0,a)$ and $b$ with $(b,0)$, then 
\begin{align*}
 [a,b] &= (0,a)(b,1)(0,a)^{-1}(b,1)^{-1} = (0,a)(b,1)(0,a')(-b,1) \\
 &=(\lambda_a(b),a)(-\lambda_{a'}(b),a') = (\lambda_a(b)-b,1) \\
 &= (a*b,1)
\end{align*}
Under this identification, we write $[X,Y]=X*Y$ for any pair of subsets $X,Y\subseteq A$. Then the iterated commutator  
satisfies 
\[
[X,\dots, X,Y]=[X,[X,\dots,[X,Y]\dots]]=L_n(X,Y),
\]
where the subset $X$ appears $n$ times.

The following theorem was proved in~\cite{mbbr} by Meng, Ballester--Bolinches and Romero
for left braces:

\begin{thm}
\label{thm:left_p}
    Let $A$ be a finite skew left brace of nilpotent type. The following statements
    are equivalent:
    \begin{enumerate}
        \item $A$ is left $p$-nilpotent.
        \item $A_{p'}*A_p=0$.
        \item The group $(A,\circ)$ is $p$-nilpotent.
    \end{enumerate}
\end{thm}

\begin{proof}
    We first prove that (1) implies (2). Since $A$ is left $p$-nilpotent,
    there exists $n\in\N$ such that 
    $L_n(A_{p'},A_p)\subseteq L_n(A,A_p)=0$.
    Since $(A_{p'},\circ)$ acts by automorphisms on $(A_{p},+)$ and this is a coprime action, it follows from~\cite[Lemma 4.29]{MR2426855} that 
    \[
    L_1(A_{p'},A_p)=A_{p'}*A_p=A_{p'}*(A_{p'}*A_p)=L_2(A_{p'},A_p).
    \]
    By induction one then proves that $A_{p'}*A_p=L_n(A_{p'},A_p)=0$.
    
    We now prove that (2) implies (3). It is enough to prove that $(A_{p'},\circ)$ 
    is a normal subgroup of $(A,\circ)$. By using Lemma~\ref{lem:factorization},  
    \[
    A_{p'}*A=A_{p'}*(A_p+A_{p'})=(A_{p'}*A_p)+(A_{p'}*A_{p'})\subseteq A_{p'}.
    \]
    since $A_p'$ is a left ideal of $A$ and $A_{p'}*A_p=0$. Then 
    $A_{p'}$ is an ideal of $A$ by Lemma~\ref{lem:Hall} and~\cite[Lemma 1.9]{csv}. In particular, $(A_{p'},\circ)$ is a normal subgroup of $(A,\circ)$. 
    
    Finally we prove that (3) implies (1). We need to prove that $L_n(A_p,A_p)=0$ for some $n$. Since $(A,\circ)$ is $p$-nilpotent, 
    there exists a normal $p$-complement that is a characteristic subgroup of $(A,\circ)$. This group is $A_{p'}$ and hence $A_{p'}$ is an ideal of $A$. Then
    $A_{p'}*A_p\subseteq A_{p'}\cap A_p=0$. We now prove that
    $L_n(A,A_p)=L_n(A_p,A_p)$ for all $n\geq0$. The case where $n=0$ is trivial, so assume that the result holds for some $n\geq0$. By the inductive hypothesis,
    \[
    L_{n+1}(A,A_p)=A*L_n(A,A_p)=A*L_n(A_p,A_p).
    \]
    Thus it is enough to prove that $A*L_n(A_p,A_p)\subseteq A_p*L_n(A_p,A_p)$. Let $a\in A$ and $b\in L_n(A_p,A_p)$. Write $a=x\circ y$ for $x\in A_p$ and $y\in A_{p'}$. Then
    \[
    a*b=(x\circ y)*b=x*(y*b)+y*b+x*b=x*b\in A_p*L_n(A_p,A_p)
    \]
    since $A_{p'}*A_p=0$. The skew left brace $A_p$ is left nilpotent
    by~\cite[Proposition 4.4]{csv}, so there exists $n\in\N$ such that $L_n(A_p,A_p)=0$.
\end{proof}

The following theorem was proved by Smoktunowicz for left braces, see~\cite[Theorem 1.1]{MR3814340}. 
For skew left braces a proof appears in~\cite[Theorem 4.8]{csv}.

\begin{thm}
    Let $A$ be a finite skew left brace of nilpotent type. Then $A$ is
    left nilpotent if and only if the multiplicative group of $A$ is nilpotent.
\end{thm}

\begin{proof}
    As it was observed in~\cite{mbbr}, Proposition~\ref{pro:left_p} and Theorem~\ref{thm:left_p} prove the theorem.
\end{proof}

\subsection*{Acknowledgments}

This work was partially supported by PICT 2016-2481 and UBACyT 20020171000256BA.
Vendramin acknowledges the support of NYU-ECNU Institute of Mathematical Sciences at NYU Shanghai. The authors thank Ferran Ced\'o and Wolfgang Rump for comments and corrections. 

\bibliographystyle{abbrv}
\bibliography{refs}

\end{document}